\documentclass[11pt,a4paper]{article}

\usepackage[top=1in, bottom=1in, left=1.25in, right=1.25in]{geometry}

\usepackage[utf8]{inputenc}
\usepackage[english]{babel}

\usepackage{amsthm}
\usepackage{amssymb}
\usepackage{amsmath}

\usepackage{enumitem}

\usepackage{graphics}

\newtheorem{theoremalpha}{Theorem}

\newtheorem{theorem}{Theorem}[section]
\newtheorem*{theorem*}{Theorem}
\newtheorem{proposition}[theorem]{Proposition}
\newtheorem{lemma}[theorem]{Lemma}
\newtheorem{corollary}[theorem]{Corollary}

\theoremstyle{definition}
\newtheorem{definition}{Definition}[section]

\DeclareMathOperator{\vol}{Vol}

\DeclareMathOperator{\mes}{Mes}
\DeclareMathOperator{\prob}{Prob}
\DeclareMathOperator{\Leb}{Leb}
\DeclareMathOperator{\grad}{grad}
\DeclareMathOperator{\supp}{supp}

\newcommand{\reals}{\mathbb{R}}

\newcommand{\TM}{T^1 M}
\newcommand{\TMu}{T^1 \tilde{M}}
\newcommand{\htil}{\tilde{h}}
\newcommand{\tH}{\tilde{\mathcal{H}}}
\newcommand{\borderM}{\partial \tilde{M}}

\newcommand{\dd}{\mathrm{d}}

\title{Unique ergodicity of horocyclic flows on nonpositively curved
	surfaces}
\author{Sergi Burniol Clotet\\
LPSM, Sorbonne Université, 4 Place Jussieu, 75005 Paris, France\\
(email: sergi.burniol\_clotet@upmc.fr)}

\date{October 27, 2022}

\begin{document}

\setlength{\parskip}{0.5ex plus 0.5ex minus 0.2ex}
\maketitle

\begin{abstract}
	On the unit tangent bundle of a nonflat compact nonpositively curved surface, we prove that there is a unique probability Borel measure invariant by a horocyclic flow which gives full measure to the set of rank $1$ vectors recurrent by the geodesic flow. If we assume in addition that the surface has no flat strips, we show that the horocyclic flow is uniquely ergodic. These results are valid for any parametrization of the horocyclic flow.
\end{abstract}

\section{Introduction}

Horocyclic flows defined on a compact hyperbolic surface are dynamical systems of high interest, for its relation with the geodesic flow and because they exhibit properties such as minimality \cite{Hedlund36} or unique ergodicity. In 1973, H. Furstenberg proved that the horocyclic flow on compact surfaces with constant negative curvature is uniquely ergodic using techniques from harmonic analysis \cite{Furstenberg73}. In 1975 it was generalized to compact surfaces with variable negative curvature by B. Marcus \cite{Marcus75}.
 We show in this article how to extend Marcus's result to the setting of nonpositve curvature. 
 
 An unstable horocycle is defined as the set of exterior normal unit vectors of a level set of a Busemann function (see Section 2.1 for a more precise definition). If the surface is orientable we can parametrize the horocycles into a flow. In this article, a horocyclic flow is any continuous flow on the unit tangent bundles whose orbits are unstable horocycles. 

\begin{theoremalpha}
	Let $M$ be an oriented nonflat compact boundaryless Riemannian surface with nonpositive curvature. Let $h_s$ be a horocyclic flow on the unit tangent bundle $T^1 M$ of $M$ and let $\Sigma_0$ denote the set of horocycles having a rank $1$ vector recurrent under the action of the geodesic flow. Then there is a unique Borel probability measure on $T^1 M$ invariant by the flow $h_s$ giving full measure to $\Sigma_0$.
\end{theoremalpha}

As we will see, the presence of flat strips, that is, immersed surfaces isometric to an Euclidean strip $[0,a]\times \reals , a>0$, poses difficulties in the study of the ergodic properties of horocyclic flows. At the end of the article we study the case of nonpositively curved compact surfaces without flat strips, which is a class of manifolds that satisfies interesting properties \cite{SC14}. On these surfaces, we prove the best result one could expect. Geodesics with vanishing curvature are still allowed in this result.

\begin{theoremalpha}
	Let $M$ be an orientable nonpositively curved compact surface without flat strips and let $h_s$ be any horocyclic flow on $T^1 M$. Then the flow $h_s$ is uniquely ergodic.
\end{theoremalpha}

The parametrization of the horocycles by their arc-length is one of the most natural horocyclic flows to consider. We call this the Lebesgue parametrization of the horocyclic flow. Nevertheless, Marcus's method relies on the Margulis parametrization. This Margulis flow $h^M_s$ is characterized by the commutative identity 
$$
g_t\circ h^M_s =h^M_{se^{\delta t}}\circ g_t,
$$
where $g_t$ is the geodesic flow on $T^1M$ and $\delta$ is its topological entropy. The Margulis flow is well defined on surfaces of strictly negative curvature. It can also be defined in the more general setting of Anosov flows with one dimensional stable leaves \cite{Marcus75}. However, there is no such parametrization on the whole unit tangent bundle of a general surface with nonpositive curvature, essentially due to the lack of hyperbolicity of the geodesic flow on some regions. But we can at least define a Margulis-like parametrization for the subset of horocycles which do not cross the non-hyperbolic regions and manage to deduce the results from this Margulis horocyclic flow. In Proposition \ref{horocyclic_flow}, we give quite optimal conditions concerning the set on which the Margulis flow can be defined.

In our previous article \cite{Burniol21}, we already proved that the Margulis horocyclic flow is uniquely ergodic.

\begin{theorem*} \cite[Theorem B]{Burniol21}
Let $M$ be an oriented nonflat compact boundaryless Riemannian surface with nonpositive curvature. Then there is a continuous Margulis horocyclic flow $h^M_s$ on the set  $\Sigma_0$ of horocycles having a rank $1$ vector recurrent under the action of the geodesic flow, and this flow is uniquely ergodic.
\end{theorem*}

This result states that there is a unique Borel probability measure on $\Sigma_0$ invariant by the Margulis flow $h_s^M$. In other words, the set $\prob(h_s^M)$ of Borel probability measures on $\Sigma_0$ invariant by the flow $h^M_s$ consists of a single element. Here our goal is to study the set $\prob(h_s^L|_{\Sigma_0})$ of Borel probability measures on $\Sigma_0$ invariant by the Lebesgue horocyclic flow $h_s^L$. We prove that there is a one-to-one correspondence between $\prob(h_s^M)$ and $\prob (h_s^L|_{\Sigma_0})$.
The proof relies on the fact that the two parametrizations are related by a continuous family of measures on the whole space, not only on the subset. Then we can generalize the unique ergodicity to any other horocyclic flow, as stated in Theorem A, thanks to the compactness of the space.

We want to clarify that there is no direct relation between Theorem A and B, since the set $\Sigma_0$ is strictly contained in $T^1 M$ even if $M$ is a nonpositively curved compact surface without flat strips. For example, we consider a surface where the curvature vanishes only on a closed geodesic as in \cite{LimaMatheusMelbourne}. A vector tangent to the closed geodesic with zero curvature has rank $2$, and all the other vectors on its horocycle are asymptotic to the closed geodesic, so these vectors are not in $\Sigma_0$.

To prove Theorem B, we will observe that, when there are no flat strips, the Margulis flow is well defined everywhere. So the proof is reduced to working with this parametrization. We can then leverage arguments from \cite{Burniol21} or \cite{Coudene1} concerning the equidistribution of horocycles under the action of the geodesic flow, to prove that horocyclic flows are uniquely ergodic.

The results of this article illustrate a dynamical property known in negative curvature which still holds in nonpositive curvature, despite the loss of uniform hyperbolicity. Our work relies on tools such as the boundary of the universal cover, the Busemann functions \cite{ballmannlectures} and the Patterson-Sullivan theory \cite{LinkPicaud,Link1}. A generalization of the classical construction of the Bowen-Margulis measure is given in \cite{Knieper98}. 
Other well known properties of the geodesic flow on negative curvature such as the uniqueness of the measure of maximal entropy and the asymptotics of the number of closed geodesics have been recently generalized in the absence of conjugate points \cite{ClimenhagaKnieperWar21a,ClimenhagaKnieperWar21b}.

Let us finally give a sketch of the organization of the paper. In Section \ref{Section2} we first recall some basic facts about rank $1$ manifolds of any dimension with nonpositive curvature, we introduce the Bowen-Margulis measure and we prove a result concerning the measure of flat strips. 
In Section \ref{Sectioncont}, we show a continuity property of the Margulis measures on the horocycles for the case of surfaces. In the remaining two sections we prove Theorems A and B. In Section \ref{Section3}, we study the relationship between invariant measures of two different parametrizations of the horocyclic flow on a subset of the unit tangent bundle and we deduce that unique ergodicity is preserved. In Section \ref{Section4} we deal with the case of nonpositively curved surfaces without flat strips.

\subsection*{Acknowledgment}

The author wishes to thank his Ph.D. advisor Yves Coudène for the many helpful suggestions.

\noindent
\begin{minipage}[c]{0.08 \linewidth}
	\includegraphics{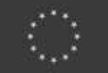}	
\end{minipage}
\begin{minipage}[c]{0.9\linewidth}
	This project has received funding from the European Union’s Horizon 2020 research and innovation program under the Marie Skłodowska-Curie grant agreement No. 754362.
\end{minipage}

\section{On the Bowen-Margulis measure and the measures on the horospheres}\label{Section2}

First, we set up the notation and terminology throughout the paper, as we explain the structure of the manifolds that we work on, \cite{ballmannlectures,ballmannnonpositivebook} are good general references on these subjects. Then we introduce the Bowen-Margulis measure on the unit tangent bundle and measures on the horospheres that are related to this measure. Under suitable hypothesis, we prove that in dimension 2 these measures depend continuously on the horocycle. We will need to apply this fact later, since the measures are involved in the definition of the Margulis parametrization.

\subsection{Geometry of nonpositively curved manifolds}

Let $M$ be a $C^\infty$ connected complete Riemannian manifold. The \textit{unit tangent bundle} $T^1M$ of $M$ is the set of unit length tangent vectors on $M$. We denote by $\pi: T^1M \to M$ the projection of tangent vectors to their base points. For a unit tangent vector $v\in T^1 M$, write $\gamma_v:\reals \to M$ for the unit speed geodesic generated by $v$, i.e. $\gamma_v'(0)=v$. The \textit{geodesic flow} $g_t:T^1M\to T^1 M$ is defined by $g_t(v)=\gamma_v'(t)$, $t\in \reals$, $v\in T^1M$.

We say that the \textit{rank of a geodesic} $\gamma$ of $M$ is the dimension of the set of parallel Jacobi fields along $\gamma$. The rank of $v\in T^1M$ is the rank of the geodesic $\gamma_v$. We will assume that the manifold $M$ has at least one geodesic of rank $1$, in other words, a geodesic whose unique parallel Jacobi field is the tangent field. Such a manifold is said to have rank $1$.  In this article, we will assume that $M$ has a closed geodesic of rank $1$. We are interested in manifolds $M$ that have nonpositive sectional curvature in every tangent plane.

The universal cover of $M$ will be denoted by $\tilde{M}$ and it is equipped with the pullback of the metric of $M$. A lift of a vector $v\in T^1 M$ to the tangent space of the universal cover $T^1 \tilde{M}$ is denoted by $\tilde v$. The same notation can be used for the lifts of objects from $M$ to $\tilde{M}$. However, the geodesic flow of $T^1 \tilde{M}$ is simply denoted by $g_t$, since no confusion can arise. Let $\Gamma$ denote the group of covering transformations of the projection from $\tilde{M}$ to $M$. The set $\Gamma$ is a subgroup of the group of isometries of $\tilde{M}$, which can also act on $T^1\tilde{M}$. The distance associated to the Riemannian metric on $\tilde{M}$ is denoted by $d$, and we use $d_1$ for the Sasaki distance on $T^1\tilde{M}$.

The universal cover $\tilde{M}$ of $M$ is a nonpositively curved simply connected complete manifold. It is then diffeomorphic to $\reals^n$, where $n$ is the dimension of $M$. 
Two geodesic rays $\sigma_1,\sigma_2: [0,+\infty)\to \tilde M$ are \textit{asymptotic} if there exists $C>0$ such that for all $t\ge0$, we have $d(\sigma_1(t),\sigma_2(t))\le C$. Since being asymptotic is an equivalence relation, we can consider the set $\partial \tilde M$ of equivalence classes of geodesic rays on $\tilde M$. Given a vector $v\in \TMu$, there is a positive and a negative geodesic ray generated by $v$, we set the notation $v_+$ and $v_-$, respectively, for the asymptotic classes of these rays in $\partial \tilde{M}$. The union $\tilde M \cup \partial \tilde M$ has a topology which extends that of $\tilde M$ and makes the space compact and the projections from $T^1M$ to $\partial \tilde M$, continuous.

Another geometric tool that is defined in nonpositive curvature are the Busemann functions. Given an asymptotic class $\xi\in\partial \tilde M$ and two points $x,y\in\tilde M$, we define 
$$
\beta_\xi(x,y)= \lim_{t\to+\infty} d(x,\sigma(t)) - d(y,\sigma(t))
$$
where $\sigma$ is a geodesic ray in the class $\xi$. This definition does not depend on the choice of the representative $\sigma $ in $\xi$. For fixed $\xi\in\partial \tilde M$ and $x\in \tilde M$, the map $\beta_\xi(x,\cdot ) : \tilde M\to \reals $ is of class $C^2$
and is called the \textit{Busemann function} at $\xi$ centered at $x$. We also know that $\beta_\xi(x,y)$ depends continuously on $(\xi,x,y)\in \partial\tilde M\times \tilde M\times \tilde M$. A horosphere in $\tilde M$ is a level set of a Busemann function $\beta_\xi(x,\cdot )$. We will rather work with horospheres in $T^1\tilde M$, which are subsets of unit vectors normal to a horosphere of $\tilde{M}$. More precisely, we define the \textit{unstable horosphere} of $v\in T^1M$ as the set 
$$
H(v)= \{-\grad_y \beta_{v_-}(\pi(v),\cdot ) \,|\,y\in \tilde M,\, \beta_{v_-}(\pi(v),y)=0\}. 
$$
All the negative geodesic rays generated by the vectors in $H(v)$ are in the same class $v_-$, i.e. $w_-=v_-$ for all $w\in H(v)$.

It is worth mentioning that horospheres are embedded $C^1$-submanifolds of $T^1\tilde M$ of dimension $n-1$. The set $\tH$ of all unstable horospheres of $T^1 \tilde M$ is a continuous foliation of $T^1\tilde M$. More precisely, if we write $B^k$ for the unit ball of $\reals^k$, for any vector $v\in T^1\tilde{M}$, there is a homeomorphism $\varphi$ from $B^{n-1} \times B^n$ to a neighborhood $U$ of $v$ such that, for every $z\in B^n$,
\begin{enumerate}[label={(\roman*)}]
	\item $\varphi(B^{n-1}\times \{z\})$ is the connected component of $ H(\varphi(0,z)) \cap U $ containing $\varphi(0,z)$,
	\item $\varphi(\,\cdot\, , z)$ is a $C^1 $ diffeomorphism onto its image,
	\item  and $z \mapsto \varphi(\,\cdot \,, z) $ is continuous in the $C^1$-topology.
\end{enumerate}
Horospheres on the quotient $T^1M$ are defined as the projection of the horospheres on $T^1 \tilde M$, and they also form a foliation $\mathcal{H}$ of $T^1 M$.

In what follows, we denote by $\tilde{R}_1$ the subset of $T^1\tilde M$ of rank $1$ vectors. The set of \textit{geodesic endpoint pairs} is
$$E(\tilde{M}):=\{(v_-,v_+)\in \partial\tilde{M}\times \borderM \,|\,v\in \TMu \}$$ and it has a subset of rank $1$ geodesic endpoint pairs
$$E_1(\tilde{M}):=\{(v_-,v_+)\in \partial\tilde{M}\times \borderM \,|\,v\in \tilde{R}_1 \}.$$ Finally, we consider
the map 
\begin{equation*}
	\begin{matrix}
		P :& T^1 \tilde M  & \longrightarrow & \borderM\times\borderM\times \reals  \\
		   & v & \longmapsto & (v_-,v_+,\beta_{v_-}(0,\pi(v))),
	\end{matrix}
\end{equation*}
where $0\in\tilde M$ is a fixed point.

The boundary of $\tilde{M}$ is useful because rank $1$ geodesics are determined by their endpoints in $\partial\tilde M$.

\begin{lemma}
		The map $P|_{\tilde{R}_1}:\tilde{R}_1\to E_1(\tilde{M})\times \reals$ is a homeomorphism.
	\end{lemma}
	
	\begin{proof}
		It is injective because if two vectors are positively and negatively asymptotic then they bound a flat strip, in particular they cannot have rank $1$. Clearly $P$ is also surjective. The continuity of $P$ and $P^{-1}$ follows from Section II.2 and Lemma III.3.1 of \cite{ballmannlectures}.
\end{proof}

\subsection{Construction of the measure}\label{Section1}
In this section, the nonpositively curved Riemannian manifold $M$ is assumed to have a closed rank $1$ geodesic. It is also assumed that $M$ is \textit{non-elementary}. This means that the limit set $\Lambda(\Gamma)\subset \partial \tilde M$ of $\Gamma$ is infinite. Recall that $\Lambda(\Gamma)$ is the set of accumulation points in $\tilde M\cup \partial \tilde M$ of an (any) orbit $\Gamma \tilde x$ of a (any) point $\tilde x\in  \tilde M$. For example, a compact nonpositively curved manifold is always non-elementary. We refer the reader who is not familiar with the Patterson-Sullivan theory and the construction of the Bowen-Margulis measure to \cite{Patterson76,Sullivan79,Roblin03}. The main results have been generalized to rank $1$ manifolds of nonpositive curvature \cite{Knieper98,Knieper02}, especially we refer to the ergodic properties and the Hopf-Tsuji-Sullivan theorem \cite{Link1,LinkPicaud}.

We briefly recall some useful facts from this theory. The Poincaré series of $\Gamma$ at $s\in \reals$ is 
$$ \sum_{\gamma\in \Gamma} e^{-sd(x,\gamma(y))}, $$
where $x$ and $y$ are two points of $\tilde M$. It is natural to wonder whether this series converges or diverges, once noted that the result does not depend on the choice of $x$ and $y$. We also define the critical exponent $\delta (\Gamma)$ of $\Gamma$ by the formula
$$ \delta(\Gamma) =\limsup_{R\to +\infty}\frac{1}{R}\log \# \{\gamma\in\Gamma \mid d(x,\gamma(y))\le R\},$$
which is a strictly positive number.
The Poincaré series converges for $s>\delta(\Gamma)$ and diverges for $s<\delta(\Gamma)$. The convergence or divergence of the series at $s=\delta(\Gamma)$ establishes two completely different behaviors of the action of the geodesic flow on $M$. We will work under the assumption that the Poincaré series diverges at $s=\delta(\Gamma)$, and we will simply say that $\Gamma$ is divergent. If the manifold $M$ is compact, then $\Gamma$ is divergent.

For $\gamma \in \Gamma$ and a measure $\sigma$ on $ \partial \tilde M$, we denote the pushforward of $\sigma$ by $\gamma$ by $\gamma_*\sigma$, which is defined by $$
\forall A\subset \partial \tilde M \text{ Borel}, \quad\gamma_*\sigma (A)= \sigma (\gamma^{-1}(A)).$$
A $\delta$-dimensional $\Gamma$-invariant conformal density is a family of measures $\{\sigma_{\tilde x}\}_{\tilde{x}\in\tilde{M}}$ on $\partial\tilde M$ with the following properties:
\begin{enumerate}[label={(\roman*)}]
\item $\forall \tilde x \in \tilde M, \, \supp\sigma_{\tilde x}\subset \Lambda(\Gamma)$
\item $\forall \tilde x, \tilde y \in \tilde M, \, \forall \xi\in\borderM,\, \frac{\dd \sigma_{\tilde y}}{\dd \sigma_{\tilde x}}(\xi)=\exp{(-\delta \beta_\xi(\tilde y, \tilde x))}$
\item $\forall \tilde x \in \tilde M, \, \forall \gamma \in \Gamma,\, \gamma_*\sigma_{\tilde x}=\sigma_{\gamma \tilde x}$
\end{enumerate}
As it is shown in \cite{Link1}, when $M$ has a closed rank $1$ geodesic and $\Gamma$ is non-elementary and divergent, there is only one $\Gamma$-invariant conformal density with dimension $\delta=\delta (\Gamma)$ (given by the Patterson construction). It is also proved that the measures $\sigma_x$ have full support in $\Lambda(\Gamma)$, so $\sigma_x(O)$ is positive for every open subset $O$ of $\partial \tilde M$ intersecting with $\Lambda(\Gamma)$. We fix a point $0\in \tilde{M}$ and consider the measure $\sigma_0$ on $\partial\tilde{M}$.

We now detail the definition of the Bowen-Margulis measure made in \cite{Knieper98} by Knieper.
For every two points $\eta,\xi\in\borderM$, let us consider the subset $\pi(P^{-1}(\{(\xi,\eta)\}\times \reals))$ of $\tilde{M}$. It is either empty, a single geodesic or a flat totally geodesic submanifold of dimension at least $2$. We denote by $\vol_{\xi,\eta }$ the induced volume measure on the submanifold $\pi(P^{-1}(\{(\xi,\eta)\}\times \reals))$.

We define an auxiliary measure $\bar{\mu}$ on $E(\tilde{M})$ by its density
\begin{equation*}
	\dd \bar{\mu}(\xi,\eta)= e^{\delta (\beta_\xi(0,p_{\xi,\eta})+ \beta_\eta(0,p_{\xi,\eta}))}\dd \sigma_0(\xi)\,\dd \sigma_0(\eta),
\end{equation*}
where $(\xi,\eta)\in E(\tilde M)$ and $p_{\xi,\eta}$ is any point in $\pi(P^{-1}(\{(\xi,\eta)\}\times \reals))$, which is non-empty. This definition does not depend on the point $0\in \tilde M$.

The \textit{Bowen-Margulis measure} is defined on the Borel $\sigma$-algebra of $T^1\tilde{M}$ by the formula 

\begin{equation}\label{definition_measure}
	\mu_{BM}(A)=\int_{E(\tilde{M})} \vol_{\xi,\eta}(\pi(P^{-1}   (\{(\xi,\eta)\}\times \reals)\cap A))\,\dd \bar{\mu}(\xi,\eta)
\end{equation}
for all Borel subsets $A\subset T^1\tilde{M}$. The measure $\mu_{BM} $ is $\Gamma$ and $g_t$-invariant, so it passes to a $g_t$-invariant measure $\mu_{BM}$ on the quotient $T^1 M$. 

Still under the assumption of divergence of $\Gamma$, it is known that the measure-preserving dynamical system $(T^1M,\mu_{BM}, g_t)$ is ergodic and conservative. Also, the action of $\Gamma$ on $\partial \tilde M$ with respect to the measure class of $\sigma_0$ is ergodic. This is part of the Hopf-Tsuji-Sullivan dichotomy for rank $1$ nonpositively curved manifolds \cite{LinkPicaud}.

\begin{proposition}\label{rank1neg}
	Let $M$ be a connected complete non-elementary manifold with nonpositive curvature. We assume that $M$ contains a rank $1$ closed geodesic 
	and that the fundamental group $\Gamma$ is divergent. Then the rank $1$ set has full Bowen-Margulis measure.
\end{proposition}
\begin{proof}
	Since $\Gamma $ is divergent, the Bowen-Margulis measure $\mu_{BM}$ is ergodic \cite{LinkPicaud}. The rank $1$ set is $g_t$-invariant, so it has either full or zero measure. The proof reduces to see that the rank $1$ set has positive measure. 
	
	Let $v\in T^1M$ be a rank $1$ $g_t$-periodic vector and $\tilde{v}\in \TMu$ one lift.  
	Since $\tilde{v}$ has rank $1$, the set $P^{-1}   (\{(\tilde{v}_-,\tilde{v}_+)\}\times \reals)$ is the geodesic of $\tilde{v}$ and $\vol_{\tilde{v}_-,\tilde{v}_+}(\pi(P^{-1}   (\{(\tilde{v}_-,\tilde{v}_+)\}\times \reals)\cap A))=\Leb (\{t\in \reals \mid g_t(\tilde{v})\in A\})$. Hence, we have
	\begin{equation*}
		\mu_{BM}(\tilde{R}_1)=\int_{E_1(\tilde{M})} \infty \,\dd \bar{\mu}(\xi,\eta).
	\end{equation*}

	We observe that the support of $\sigma_0$ is the limit set $\Lambda(\Gamma)$, which is a consequence of the minimality of the action of $\Gamma$ on $\Lambda (\Gamma)$ \cite[Proposition 2.8]{Ballmann82}. We also know that $E_1(\tilde{M})$ is open \cite{ballmannlectures}: there exist disjoint open neighborhoods $U_-$ and $U_+$ of $\tilde{v}_-$ and $\tilde{v}_+$ in $\borderM$, respectively, such that $U_-\times U_+ \subset E_1(\tilde{M})$. These neighborhoods have positive $\sigma_0$-measure because $\tilde{v}_-$ and $\tilde{v}_+$ are in $\Lambda(\Gamma)=\supp \sigma_0$ because $v$ is periodic. This implies that the $\bar{\mu}$-measure of $E_1(\tilde{M})$ is positive, so $\mu_{BM}(\tilde{R}_1)=\infty$. 
	When $\tilde{R}_1 $ is projected to the quotient $T^1M$, it still has positive measure.
\end{proof}
The hypothesis on the existence of a closed rank $1$ geodesic is necessary for the Patterson-Sullivan theory to be relevant. Let $f:\TMu \to \reals$ be a measurable function.
The integral of $f$ with respect to $\mu_{BM}$ has a simpler expression now. For $(\xi,\eta)\in E_1(\tilde{M})$ and $t\in \reals$ we denote $f(\xi,\eta ,t):=f(P^{-1}(\xi,\eta ,t))$. The next result says that the volume factors of the Bowen-Margulis measure that appear in (\ref{definition_measure}), and which are used in the cited works, can be forgotten.

\begin{corollary} Let $f:\TMu\to \reals$ be a Borel function. We have
	\begin{equation*}
		\int_{\TMu}f \,\dd \mu_{BM} =\int_{E_1(\tilde{M})} \int_\reals f(\xi,\eta,t) \,\dd t \, \dd \bar{\mu}(\xi,\eta).
	\end{equation*}
	
\end{corollary}

Let $H$ be a horosphere in $\TMu$. It can be described as $H=\{w\in \TMu \mid w_-=v_-,\,\beta_{v_-}(\pi(v),\pi(w))=0\}$, where $v$ is any vector in $H$. We consider the map $P_H:H\to \borderM,\, v\mapsto v_+$, which is identified to the restriction of $P$ to $H$. For every $\xi\in\borderM$, let us consider the subset $P_H^{-1}(\xi)$ of $H$, which can be either empty, a single point or a totally geodesic submanifold of $H$. We denote by $\vol_{\xi,H }$ the induced volume measure of the submanifold $\pi(P_H^{-1}(\xi))$.

If $\eta \in P_H(H)$, we choose $w\in P_{H}^{-1}   (\eta)$ and write $\phi_H(\eta)=e^{\delta(\Gamma) \beta_{\eta}(0,\pi(w))}$, which in fact only depends on $\eta$, but not on $w$. We set $\phi_H(\eta)=0$ if $\eta\in\borderM\setminus P_H(H)$. We define a measure $\mu_H$ on $H$ by

\begin{equation*}
	\mu_{H}(A)=\int_{\partial\tilde{M}} \vol_{\eta,H}(\pi (P_{H}^{-1}(\eta)\cap A))
	\phi_H(\eta)
	\,\dd \sigma_0(\eta),
\end{equation*}
where $A$ is a Borel subset of $H$. This is how we defined a family of measure $\{\mu_H\}_{H\in \tH}$ in \cite{Burniol21}, which are exponentially expanded by the geodesic flow, i.e. they satisfy $$\mu_{g_tH}=e^{\delta t}g_{t*}\mu_{H}.$$ 
The $\Gamma$-invariance of the conformal density $\{\sigma_x\}_{x\in\tilde M}$ yields the $\Gamma$-invariance of the measures on the horospheres: $\gamma_* \mu_H =\mu_{\gamma H}, \, \forall \gamma\in \Gamma$.

As in the Bowen-Margulis measure, the volume factor can be removed if the integration is restricted to an appropriate subset. This simplification is a consequence of the next result.
We define the \textit{singular subset} of $\TMu$ as the set $\tilde{S}=\TMu\setminus \tilde{R}_1$, which is invariant by the geodesic flow and by $\Gamma$. We write $\tilde{S}_+=\{v_+\in \borderM\mid v\in \tilde{S}\}$.

\begin{proposition}
	Let $M$ be a connected complete non-elementary manifold with nonpositive curvature. We assume that $M$ contains a rank $1$ closed geodesic and that the fundamental group $\Gamma$ is divergent. Then the set $\tilde{S}_+$ is $\sigma_0$-negligible.
\end{proposition}
\begin{proof}
	In \cite[Proposition 4]{LinkPicaud} it is shown that $\Gamma$ acts ergodically on $\borderM$ when $\Gamma $ is divergent. Since $\tilde{S}$, so $\tilde{S}_+$, is $\Gamma$-invariant, then $\sigma_0(\tilde{S}_+)=0$ or $\sigma_0(\tilde{S}_+)=\sigma_0(\borderM)$.
	
	The system $(T^1M,\mu_{BM}, g_t)$ is conservative by \cite[Proposition 4]{LinkPicaud}. Thanks to the Poincaré recurrence theorem \cite[1.1.5]{AaronsonIET}, since $T^1M$ is a second-countable metric space, $\mu_{BM}$-almost every vector of $T^1 M$ is positively and negatively $g_t$-recurrent. The lift $\widetilde{Rec}$ of the set of recurrent points to $\TMu$ is a subset of full $\mu_{BM}$-measure.
	
	We have seen in Proposition \ref{rank1neg} that $\tilde{R}_1$ is also a subset of full measure. Then the intersection $\tilde{R}_1\cap \widetilde{Rec}$ has also full measure. We denote by $E\subset \borderM$ the set of positive endpoints of rank 1 vectors in $T^1 \tilde{M}$ such that their projections to $T^1M$ are both positively and negatively recurrent under the geodesic flow, i.e. $E= \{ v_+ \mid v\in \tilde{R}_1\cap \widetilde{Rec} \}$. The $\sigma_0$-measure of $E$ has to be positive because $P(\tilde{R}_1\cap \widetilde{Rec})\subset E\times E\times \reals$.
	
	Finally, we observe that $E$ and $\tilde{S}_+$ are disjoint: if $v$ is a rank $1 $ recurrent vector, all the vectors on its stable horosphere have rank $1$ \cite[Lemma 2.3]{Burniol21}, so all the vectors pointing to $v_+$ have rank $1$ also, which means that $v_+\notin \tilde{S}_+$.
	We conclude that the complement of $\tilde{S}_+$ has positive measure, so $\tilde{S}_+$ is $\sigma_0$-negligible.
\end{proof}

\begin{corollary} \label{simplification_measure}
	Let $H$ be a horosphere in $\TMu$ and $A$ a Borel subset of $H$. Then
	\begin{equation*}
		\mu_{H}(A)=\int_{P_H(H)\setminus \tilde{S}_+} \mathbf{1}_{A}(P_H^{-1}(\eta))
		\phi_H(\eta)
		\,\dd \sigma_0(\eta)
	\end{equation*}
	\begin{equation*}
		=\int_{P_H(H\cap \tilde{R}_1)} \mathbf{1}_{A}(P_H^{-1}(\eta))
		\phi_H(\eta)
		\,\dd \sigma_0(\eta),
	\end{equation*}
\end{corollary}
\begin{proof}
	On the one hand, $\phi_H$ vanishes outside $P_H(H)$ and $\tilde{S}_+$ has $0$ measure, so we have
	\begin{equation*}
		\int_{\partial\tilde{M}} \vol_{\eta,H}(\pi (P_{H}^{-1}(\eta)\cap A))
		\phi_H(\eta)
		\,\dd \sigma_0(\eta)=		\int_{P_H(H)\setminus \tilde{S}_+} \! \! \! \! \vol_{\eta,H}(\pi (P_{H}^{-1}(\eta)\cap A))
		\phi_H(\eta)
		\,\dd \sigma_0(\eta).
	\end{equation*}
	On the other hand, for $\eta \in P_H(H\cap \tilde{R}_1)$, the set $P_H^{-1}(\eta )$ consists of a single point, so the volume factor is just the indicator function of $A$. Since $P_H(H)\setminus \tilde{S}_+ \subset P_H(H\cap \tilde{R}_1)$, we obtain
	\begin{equation*}
		\mu_H(A)	=	\int_{P_H(H)\setminus \tilde{S}_+} \mathbf{1}_A(P_{H}^{-1}(\eta))
		\phi_H(\eta)
		\,\dd \sigma_0(\eta).
	\end{equation*}
	The second equality is true because $P_H(H)\setminus \tilde{S}_+$ has full measure in $P_H(H\cap \tilde{R}_1)$.
\end{proof}
An immediate consequence is that $H\cap \tilde{R}_1 $ is of full $\mu_H$-measure in $H$. Finally, we recall some known conditions for a rank $1$ manifold to have a closed rank $1$ geodesic.

\begin{proposition}\label{conditions}
	\cite{ballmannlectures,ballmannnonpositivebook} Let $M$ be a connected complete non-elementary rank $1$ Riemannian manifold of nonpositve curvature.
	If we assume one of the following:
	\begin{itemize} 
		\item the limit set $\Lambda (\Gamma)$ of $\Gamma$ is equal to the whole boundary $\borderM$,
		\item $M$ has finite Riemannian volume;
	\end{itemize}
	then $M$ has a closed rank $1$ geodesic.
\end{proposition}

\subsection{Continuity of the measures on the horocycles}\label{Sectioncont}

Let $M$ be a connected complete non-elementary Riemannian surface of nonpositve curvature with a closed rank $1$ geodesic. We assume additionally that the group $\Gamma$ is divergent. The next result expresses the continuity of the measure $\mu_H$ with respect to the horocycle $H$, and is proved from the simplified expressions of the measures.

A horocycle $H$ in $\TMu$ is diffeomorphic to $\reals$, so it makes sense to speak of the \textit{open interval} $(v,w)$ of $H$ between the vectors $v$ and $w$ of $H$. In the same way we can define the closed interval $[v,w]$, the half-open intervals $[v,w),\ (v,w]$.

\begin{proposition}\label{continuity}
	The map
	\begin{equation*}
		\begin{matrix}
			\{(v,w)\in \TMu \times \TMu \,|\,w\in H(v)\} & \longrightarrow & \reals \\
			(v,w) & \longmapsto & \mu_{H(v)}((v,w))
		\end{matrix}
	\end{equation*}
	is continuous.
\end{proposition}

\begin{proof}
	
	The measure of the interval $(v,w)$ is

	\begin{equation*}
		\mu_{H(v)}((v,w))=\int_{P_{H(v)}((v,w))\setminus \tilde{S}_+} \mathbf{1}_{(v,w)}(P_{H(v)}^{-1}(\eta))
		\phi_{H(v)}(\eta)
		\,\dd \sigma_0(\eta).
	\end{equation*}
	The set $P_{H(v)}((v,w))$ is an interval of $\borderM$ that satisfies $(v_+,w_+)\subset P_{H(v)}((v,w))\subset [v_+,w_+]$. Since $\sigma_0$ has no point masses \cite[Proposition 5]{LinkPicaud},
	 we can write
	\begin{equation*}
		\mu_{{H(v)}}((v,w))=\int_{(v_+,w_+)\setminus \tilde{S}_+} 
		\phi_{H(v)}(\eta)
		\,\dd \sigma_0 (\eta).
	\end{equation*}
	
	Let $f:\TMu \to \reals$ be the real continuous function given by $f(u)=\exp(\delta \beta_{u_+}(0,\pi(u)))$. 
	Because of the definition of $\phi_{H(v)}$ and the fact that $(v_+,w_+)\setminus \tilde{S}_+\subset P_{H(v)}(H(v)\cap \tilde{R}_1)$, we have 
	$$
	\forall \eta \in (v_+,w_+)\setminus \tilde{S}_+, \quad \phi_{H(v)}(\eta )= e^{\delta \beta_{\eta}(0,\pi(P_{H(v)}^{-1}(\eta)))} = f(P_{H(v)}^{-1}(\eta)).
	$$
	
	Let $v\in \TMu$ and $w\in H(v)$. We want to show the continuity of the map in the statement at $(v,w)$. Let us explain step by step the needed estimates and at the end we will apply them. We fix $\varepsilon>0$.
	
	\textbf{Boundedness of the integrated function.} Let $K$ be a compact neighborhood of $ [v,w]$ in $\TMu$. By the continuity of horocycles there are open neighborhoods $V_1\subset K$ and $W_1\subset K$ of $v$ and $w$ in $\TMu$ such that if $v'\in V_1 $ and $w'\in W_1 \cap H(v')$ the interval $(v',w')$ is contained in $K$. The function $f$ is bounded on $K$ by a constant $C>0$.
	
	For all $v'\in V_1$, for all $w'\in W_1\cap H(v')$, we have the inclusion
	$$
	P_{H(v')}^{-1}((v'_+,w'_+)\setminus \tilde{S}_+)\subset (v',w') \subset K,
	$$
	which says that, for all $\eta \in (v'_+,w'_+)\setminus \tilde{S}_+$, the quantity $f(P_{H(v')}^{-1}(\eta))$ is bounded by $C$.
	
	\textbf{Approximation of intervals.} Since $\sigma_0$ is finite and $\partial \tilde M$ is a metrizable 
	Since $\sigma_0$ has no point masses and is outer regular \cite[Proposition 18.3]{CoudeneBook}, there are open intervals $A$ and $B$ around $v_+$ and $w_+$ in $\borderM$ of arbitrarily small measure. We choose $A$ and $B$ such that $\sigma_0(A)<\varepsilon/16C$ and $\sigma_0(B)<\varepsilon/16C$. By the continuity of the projection $T^1 \tilde{M} \to \borderM$ to the endpoint, there are neighborhoods $V_2$ and $W_2$ of $v$ and $w$ such that ${V_2}_+\subset A$ and ${W_2}_+\subset B$. For every $v'\in V_2 $ and $w'\in W_2$, we have
	$$
	(v_+,w_+)\triangle(v'_+,w'_+)\subset [v_+,v'_+]\cup [w_+,w'_+]\subset A\cup B,
	$$
	so $\sigma_0((v_+,w_+)\triangle(v'_+,w'_+))\le \sigma_0(A)+\sigma_0(B)<\frac{\varepsilon}{8C}$ (Figure \ref{fig:pic1}).
	
	\begin{figure}[h]
		\centering
		\includegraphics{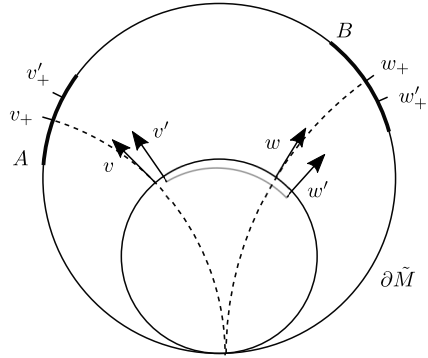}
		\caption{ \label{fig:pic1}
			Positions of the endpoints of the vectors $v,\,w,\,v' $ and $w'$.}
	\end{figure}
	
	\textbf{Continuity of the integrated function uniform with respect to $\eta$.} First, we apply the inner regularity of $\sigma_0$ \cite[Theorem 18.2]{CoudeneBook}: there exists a compact subset $F\subset (v_+,w_+)\setminus \tilde{S}_+$ such that $ \sigma_0(((v_+,w_+)\setminus \tilde{S}_+)\setminus F)<\varepsilon/8C$. 
	
	If $v' \in\TMu$ and $\eta\in P_{H(v')}(H(v')\cap \tilde{R}_1)$ then $(v'_-,\eta )\in E_1(\tilde{M})$. Since $F\subset (v_+,w_+)\setminus \tilde{S}_+\subset P_{H(v)}(H(v)\cap \tilde{R}_1)$, one has that $\{v_-\}\times F$ is a subset of $E^1(\tilde{M})$. The subset $E^1(\tilde{M})$ is open in $\borderM\times \borderM$ and $\{v_-\}\times F$ is compact, so there exists a neighborhood $A$ of $v_-$ in $\borderM$ such that $A\times F\subset E_1(\tilde{M})$. Let $U$ be a neighborhood of $v$ such that $U_+$ is contained in $A$. Now there is a well defined map
	\begin{equation*}\label{interesting_map}
		\begin{matrix}
			U\times F & \longrightarrow & \TMu  \\
			(v',\eta) & \longmapsto & P_{H(v')}^{-1}(\eta)=P^{-1}(v'_-,\eta,\beta_{v'_-}(0,\pi(v'))),
		\end{matrix}
	\end{equation*}	
	which is continuous, because $P^{-1}$ is continuous on $E_1(\tilde{M})\times \reals$ and the Busemann function depends continuously on its three variables.
	Composing by $f$ we obtain a continuous map from $U\times F$ to $\reals$. Since $F$ is compact, this map is continuous at $v$ uniformly with respect the second variable: there exists a neighborhood $U_0\subset U$ of $v$ such that, for all $v'\in U_0$, for all $\eta \in F$,
	$$
	|f(P^{-1}_{H(v')}(\eta )) - f(P^{-1}_{H(v)}(\eta ))|<\frac{\varepsilon}{2\sigma_0((v_+,w_+))}.
	$$
	
	\textbf{Conclusion.} 
	We choose a subset $F$ of $(v_+,w_+)\setminus \tilde{S}_+$ as explained above. Thanks to the integral expressions of the measures we can write,
	for $v'\in\TMu$ and $w'\in H(v')$,
	\[
	|\mu_{H(v)}((v,w))-\mu_{H(v')}((v',w'))|\le 
	\int_{F \cap (v'_+,w'_+)} 
	| f(P_{H(v)}^{-1}(\eta)) - f(P_{H(v')}^{-1}(\eta)) |
	\,\dd \sigma_0(\eta)
	\tag{$\ast$}
	\]
	\[		+ 
	\int_{(((v_+,w_+)\setminus F) \cup  ((v_+,w_+) \triangle (v'_+,w'_+)))\setminus \tilde{S}_+} 
	(| f(P_{H(v)}^{-1}(\eta)) |+ | f(P_{H(v')}^{-1}(\eta))|) 	 
	\,\dd \sigma_0(\eta) .
	\tag{$\ast\ast$}
	\]
	We can now see that both terms are small if $(v',w')$ is close to $(v,w)$.
	
	We set $V=U_0 \cap V_1\cap V_2 $ and $W=W_1\cap W_2$. Let $v'\in V$ and $w'\in W\cap H(v')$, so that we can apply all the bounds found above. In the first integral ($\ast$), the integrand is bounded by $\frac{\varepsilon}{2\sigma_0((v_+,w_+))}$ and the integrating set has measure at most $\sigma_0((v_+,w_+))$ (we suppose $\sigma_0((v_+,w_+))>0$ because if $\sigma_0((v_+,w_+))=0$ the integral is trivially $0$). In ($\ast\ast$), the integrand is bounded by $2C$ and the measure of the integrating set is at most $\sigma_0((v_+,w_+)\setminus F)+ \sigma_0((v_+,w_+)\triangle (v'_+,w'_+))<\varepsilon/4C$. So the result is less than $\varepsilon$.
	
\end{proof}

\section{Reparametrization of the horocyclic flow}\label{Section3}
Next we address the ergodic properties of the horocyclic flow. When one wants to study these properties, it often arises the question of the parametrization of the horocycles: while it is more natural to work with the Riemannian length parametrization, it may be required to use a parametrization more adapted to the dynamics of the flow. In \cite{Burniol21} we proved a result of unique ergodicity of the horocyclic flow using the Margulis parametrization. In this section we will study how invariant measures are transformed by a change of parametrization. This will eventually allow us to conclude that the unique ergodicity also holds for the arc-length parametrization. We start by a general result on reparametrization of flows.

\subsection{A general reparametrization result}

Let $X$ be a topological space. A flow on $X$ is a map 
\begin{equation*}
\begin{matrix}
 X\times \reals & \longrightarrow & X  \\
 (x,t) & \longmapsto & f_t(x).
\end{matrix}
\end{equation*}
satisfying for all $x\in X$ and $t_1,t_2\in  \reals$, $f_0(x)=x$ and $f_{t_1+t_2}(x)=f_{t_1}(f_{t_2}(x))$.
By abuse of notation, we denote the flow by $f_t$. We say that the flow $f_t$ is continuous if the map $(x,t)\mapsto f_t(x)$ is continuous. The orbit of $x\in X$ is the set $ \{f_t(x)\}_{t\in\reals}$. The point $x$ is fixed by the flow $f_t$ if its orbit consists of a single point.

Let $f_t$ be a continuous flow on $X$. A Borel measure $\mu$ on $X$ is said to be invariant by $f_t$ if for every Borel subset $A$ of $X$ we have, for all $t\in  \reals$, $\mu (f_t(A))=\mu(A)$. Let $\mes(f_t)$ denote the set of locally finite Borel measures on $X$ invariant by $f_t$. Let $\prob(f_t)$ be the subset of $\mes(f_t)$ consisting of the probability measures. The flow $f_t$ is \textit{uniquely ergodic} if the set $\prob(f_t)$ consists of a single measure. The next result is due to M. Beboutoff and W. Stepanoff in 1940.

\begin{theorem} \cite{BeboutoffStepanoff} \label{BebStep}
	Let $X$ be a separable metric space. Let $f_t$ be a continuous flow on $X$ without fixed points. Let $g_t$ be another continuous flow on $X$ with the same orbits than $f_t$, i.e. for all $x\in X$, we have $\{g_t(x)\}_{t\in \reals}=\{f_t(x)\}_{t\in \reals} $. Then there is a bijective correspondence $\Phi: \mes(f_t)\to\mes(g_t)$.
\end{theorem}
The correspondence $\Phi$ will be described in the next section in the setting of horocyclic flows.
If $X$ is a compact space, since every locally finite measure is in fact finite, the map $\Phi$ given in the theorem is an isomorphism between the spaces of finite Borel invariant measures of $f_t$ and $g_t$. This induces a bijection between $\prob(f_t)$ and $\prob(g_t)$. In particular, $f_t$ is uniquely ergodic if and only if $g_t$ is uniquely ergodic.

However, in a general separable metric space, we think that there is no reason why the map of the theorem would send finite measures to finite measures, and infinite measures to infinite measures.

\subsection{Reparametrization of the horocyclic flow on a compact surface}

In this section, we apply the ideas of Beboutoff and Stepanoff \cite{BeboutoffStepanoff} to establish what happens to a measure invariant by a horocyclic flow under a change of parametrization. We specify the bijection $\Phi$ for a horocyclic flow on a compact surface.

In what follows, let $M$ denote an oriented compact connected rank $1$ surface with nonpositive curvature. Such a manifold $M$ is complete, has a rank $1$ closed geodesic, and the covering transformations group $\Gamma$ is non-elementary and divergent \cite[Theorem 4.3]{Knieper98}. Moreover, the Bowen-Margulis measure we defined in Section \ref{Section1} is finite.

By \textit{horocyclic flow} we mean a continuous flow $h_s$ on $T^1M$ whose orbits are the unstable horocycles, i.e $$\forall v\in T^1M,\,\{h_s(v)\}_{s\in\reals}=H(v).$$ 
\begin{definition}
	The Lebesgue horocyclic flow $h^L_s$ is given by the arc length of the horocycles: for any $v\in T^1 M$ and $s\in \reals$, the vector $h^L_s(v)$ is the vector of $H^u(v)$ that we get by travelling a distance $|s|$ on $H^u(v)$ from $v$ in the positive or negative direction accordingly to the orientation depending on the sign of $s$. We will say that the horocyclic flow  $h^L_s$ has the \textit{Lebesgue parametrization}.
\end{definition}

A horocyclic flow $h_s$ on $T^1 M$ is lifted to a horocyclic flow $\tilde{h}_s$ on the unit tangent bundle  $T^1 \tilde{M}$ of the universal cover satisfying $\tilde{h}_s(\gamma v)=\gamma \tilde{h}_s(v)$ for every $s\in \reals,\, v\in \TMu,\,\gamma \in \Gamma$. Conversely, a horocyclic flow $\htil_s$ on $T^1\tilde{M}$ with the property $\tilde{h}_s(\gamma v)=\gamma \tilde{h}_s(v)$ passes to the quotient $T^1 M$ giving a horocyclic flow $h_s$. Moreover, ${h}_s$-invariant measures on $T^1M$ are in correspondence with $\tilde{h}_s$-invariant measures on $T^1\tilde{M}$ which in addition are $\Gamma$-invariant.

The \textit{weak unstable manifold} $\tilde{W}^{wu}(v)$ of a vector $v\in T^1\tilde M$ is the union of all horocycles along the geodesic generated by $v$, 
$$\tilde{W}^{wu}(v)=\cup_{t\in \reals}H(g_t v)= \{w\in T^1 \tilde M \mid w_-=v_-\},$$ 
and has dimension $2$. In the universal cover, the \textit{weak stable manifold} $\tilde{W}^{ws}(v)=-\tilde{W}^{wu}(-v)$ of a vector $v\in \TMu$ is a section on the flow in the sense of \cite{BeboutoffStepanoff} in many cases, as explained in the next lemma.

\begin{lemma}\label{section}
	Let $v\in \TMu$. Assume that the weak stable manifold $\tilde{W}^{ws}(v)$ of $v$ has no rank $2$ vectors. Then for every $w\in \TMu\setminus \tilde{W}^{wu}(-v)=\{u\in \TMu\,|\,u_-\neq v_+\}$, there exists a unique time $s\in \reals$ such that $\tilde{h}_s(w)\in \tilde{W}^{ws}(v)$.
\end{lemma}

\begin{proof}
	It is known that a nonflat compact surface $M$ with nonpostive curvature satisfies the visibility axiom, which means that any two distinct points of the boundary $\partial \tilde M$ can be joined by a geodesic on $\tilde M$ \cite[Proposition 2.5]{Eberlein79}. If $v$ and $w$ are as in the statement, the points $v_+$ and $w_-\in \partial \tilde M$ are distinct, so there is at least a geodesic between them. Hence, there exists a vector $u$ in the unstable horocycle $H(w)$ of $w$ pointing to $v_+$. This vector can be written as $u=\tilde{h}_s(w)$ for some $s\in \reals$ and is in $\tilde{W}^{ws}(v)$.
	
	To prove the uniqueness, let us suppose that for some different reals $s$ and $s'$, $\tilde{h}_s(w)$ and $\tilde{h}_{s'}(w)$ are in the weak stable manifold $\tilde{W}^{ws}(v)$. Then the vectors $\tilde{h}_s(w)$ and $\tilde{h}_{s'}(w)$ are asymptotic both for positive and negative time, so the corresponding geodesics bound a flat strip. This would imply that these vectors have rank $2$, which contradicts the hypothesis that $\tilde{W}^{ws}(v)$ has no such vectors.
\end{proof}

The condition that $\tilde{W}^{ws}(v)$ has no rank $2$ vectors is equivalent to the fact that $v_+$ is not in the set $\tilde{S}_+$ of endpoints of vectors of rank $2$. We know that $\tilde{S}_+$ has zero $\sigma_0$-measure, so its complement must be dense in $\partial \tilde M$. This ensures that there are enough sections for the horocyclic flow.

We recall how an invariant measure is locally disintegrated.
Let $\mu\in \mes(\htil_s)$ be an invariant measure. Given a Borel subset $A$ of a section $\tilde{W}^{ws}(v)$, we consider the function $\phi_A:[0,1]\to \reals^+\cup \{+\infty\}$ defined by $$\phi_A(s)=\mu(\htil_{[0,s]}(A)).$$
Let us assume that $\phi_A(1)$ is finite. Then we have, for any integer number $n\ge 1$,
$$
	\phi_A(0)=\mu(A)\le \mu(\tilde{h}_{[0,1/n)}(A))=\frac{1}{n}\sum_{k=0}^{n-1} \mu(\tilde{h}_{k/n}(\tilde{h}_{[0,1/n)}(A)))=\frac{1}{n}\mu(\tilde{h}_{[0,1)}(A))\le \frac{\phi_A(1)}{n}, 
$$
hence $\phi_A(0)=0$. 
Moreover, for every two nonnegative numbers $s,t$ such that $s+t\le 1$, we have $$\phi_A(s+t)=\phi_A(s) +\mu(h_s(h_{(0,t]}(A)))=\phi_A(s)+\phi_A(t)-\phi_A(0)=\phi_A(s)+\phi_A(t)$$ thanks to the invariance of $\mu$. Since $\phi_A$ is monotonic, we deduce that it is linear, so there is a constant $l_A\ge0$ such that $\phi_A(t)=l_A t$ for all $t\in [0,1]$. 

We now define a measure $\mu_{\tilde{W}^{ws}(v)}$ on $\tilde{W}^{ws}(v)$ which associates the value $l_A=\phi_A(1)$ to the set $A$ if $\phi_A(1)$ is finite, and the value $\infty$ otherwise. It is not difficult to check that $\mu_{\tilde{W}^{ws}(v)}$ is a Borel locally finite measure and it is the same for two vectors on the same weak stable leaf. Furthermore, the measure $\mu$ is the product of the Lebesgue measure on each horocycle by the measure $\mu_{\tilde{W}^{ws}(v)}$ (Figure \ref{fig:picC}): for every Borel subset $E\subseteq \TMu\setminus \tilde{W}^{wu}(-v)$ we have
$$
	\mu(E)=\int_{\tilde{W}^{ws}(v)}\int_\reals \mathbf{1}_E(\tilde{h}_s(u)) \, \dd s \,\dd \mu_{\tilde{W}^{ws}(v)}(u).
$$

\begin{figure}[h]
	\centering
	\includegraphics{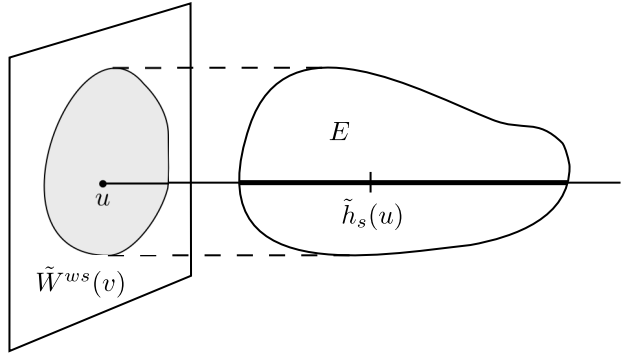}
	\caption{Decomposition of $\mu$.}
	\label{fig:picC}
\end{figure}

Let $h'_s$ be another horocyclic flow of $T^1M$. The two horocyclic flows $\htil_s$ and $\htil'_s$ on $T^1\tilde{M}$ are related by a change of time, given $v\in \TMu$ and $s'\in\reals$, there exists a unique $s=s(s',v)\in\reals$ such that $\htil_s(v)=\htil'_{s'}(v)$. In fact, $s$ as a function $\reals \times \TMu \to \reals$ satisfies:
\begin{enumerate}[label={(\roman*)}]
	\item For all $v\in T^1\tilde M$ and $s'\in\reals$, we have $\htil'_{s'}(v)=\htil_{s(s',v)}(v)$.
	\item $s$ is continuous.
	\item For all $v\in\TMu,\, s(\,\cdot \,,v):\reals \to \reals$ is strictly monotonic.
	\item For all $v\in T^1\tilde M$, $s(0,v)=0$.
	\item For all $v\in T^1\tilde M$ and $s_1',s_2'\in\reals$, $s(s'_1+s'_2,v)=s(s'_1,v)+ s(s'_2,\htil_{s(s'_1,v)}(v))$.
	\item For all $v\in T^1\tilde M$, $s'\in\reals$ and $\gamma\in\Gamma$, $s(s',\gamma v)=s(s',v)$
\end{enumerate}

The converse is also true: given a function $s$ with the properties (ii)-(vi) we can define a new horocyclic flow by $\htil'_{s'}(v) :=\htil_{s(s',v)}(v)$ which passes to the quotient $T^1M$.

%

As stated in Theorem \ref{BebStep}, there is a correspondence between the measures invariant by the two flows.

\begin{definition}
	Given two horocyclic flows $\tilde h_s$ and $\tilde h'_s$ of $T^1 \tilde M$, we define a map $\Phi:\mes(\tilde h_s)\to \mes (\tilde h'_s)$ putting, for every weak stable manifold $\tilde W^{ws}(v) $ containing no rank $2$ vectors and any Borel subset $E$ of $T^1\tilde M\setminus \tilde W^{wu}(-v)$,
	\begin{equation}\label{muprime}
		(\Phi(\mu))(E)=\int_{\tilde{W}^{ws}(v)}\int_{\reals} \mathbf{1}_E(\tilde{h}'_s(u)) \, \dd s \,\dd \mu_{\tilde{W}^{ws}(v)}(u).
	\end{equation}
\end{definition}

By \cite{BeboutoffStepanoff}, the measure $\mu' =\Phi(\mu)$ is well defined and $\tilde h'_s$-invariant. 
If the measure $\mu$ is in addition $\Gamma$-invariant, then so is the measure $\mu'$. We obtain, after normalization of measures, a correspondence between the sets $\prob({h}_s)$ and $\prob(h'_s)$, because $T^1M$ is compact. We also remark that the measures $\mu_{\tilde{W}^{ws}(v)}$ are independent of the parametrization of the horocyclic flow.

\subsection{Horocyclic flows on a subset of $T^1M$}

Let $\Sigma$ be Borel subset of $\TM$ which is a union of horocycles and let $\tilde{\Sigma}$ be its lift to $\TMu$.  
We want to study the horocyclic flows on $\Sigma$, namely, the continuous flows defined on $\Sigma$ whose orbits are horocycles. It is clear that a horocyclic flow $h_s$ on $\TM$ is restricted to a horocyclic flow $h_s|_\Sigma$ on $\Sigma$, and that a measure $\mu \in \mes(h_s)$ can be restricted to a measure $\mu|_\Sigma\in \mes(h_s|_\Sigma)$. In many situations, a certain parametrization of the horocyclic flow is just defined on a subset $\Sigma$ and we would like to deduce ergodic properties for the flow on the whole space from this specific parametrization.

The set $\Sigma$ does not have to be compact, or even locally compact. Then, if we have two horocyclic flows $h_s$ and $h'_s$ on $\Sigma$, we still have a bijection between $\mes(h_s)$ and $\mes(h'_s)$ by Theorem \ref{BebStep}, but we have no information about the subsets $\prob(h_s)\subset \mes(h_s)$ and $\prob(h'_s)\subset \mes(h'_s)$ or the relation between them.


We can restrict the sections of the horocyclic foliation that we found in the previous section to $\tilde \Sigma$ and disintegrate invariant measures with respect to them. More precisely, if the set $\tilde{W}^{ws}(v)\cap \tilde{\Sigma}$ is nonempty and $\tilde{W}^{ws}(v)$ has no rank $2$ vectors, then $\tilde{W}^{ws}(v)\cap \tilde{\Sigma}$ is a section for the flow $\htil_s$ and we can define a measure $\mu_{\tilde{W}^{ws}(v)\cap \tilde{\Sigma}}$ on this set from a $\tilde h_s$-invariant measure $\mu$ on $\tilde \Sigma$. 

Another horocyclic flow $\htil'_{s'}$ on $\tilde{\Sigma}$ is related to $\htil_s$ by a change of time $s=s(s',v)$ as before. The $\htil'_s$-invariant measure $\mu'= \Phi(\mu)$ associated to $\mu\in \mes(\htil_s)$ has a local expression given by Equation (\ref{muprime}) as a product of the Lebesgue measures on horocycles by $\mu_{\tilde{W}^{ws}(v)\cap \tilde{\Sigma}}$.

\subsection{The Margulis parametrization}

Some of the most relevant properties of the horocyclic flow are deduced thanks to the Margulis parametrization, which allows to apply the usual techniques of ergodic theory.  The situation in nonpositive curvature is different from negative curvature, because the Margulis parametrization can only be defined on a certain subset $\Sigma_0$ of $T^1M$.
As we mentioned, our result \cite{Burniol21} establishes the unique ergodicity of the horocyclic flow in restriction to the subset $\Sigma_0$, when $M$ is a nonpositively curved rank $1$ compact surface. Here we will study the relation between the Margulis parametrization $h_s^M$ on $\Sigma$ and the Lebesgue parametrization $h_s^L$ on $T^1M$. We will see that the map $\Phi$ induces a bijection between $\prob(h^L_s|_\Sigma)$ and $\prob(h^M_s)$, which is not trivial at all because no assumptions of compactness are made on $\Sigma$.

As we have seen in Section \ref{Section1}, the set of horocycles $\tH$ of $\TMu$ admits a family of measures $\{\mu_H\}_{H\in\tH}$, which is exponentially expanded by the geodesic flow, $$\mu_{g_tH}=e^{\delta t}g_{t*}\mu_{H}.$$ To define the Margulis parametrization, we parametrize each horocycle $H$ by the measure $\mu_H$. Let us explore some further properties of these measures. 

\begin{lemma}\label{point mass}
	The measures $\mu_H$ are locally finite and have no point masses.
\end{lemma}

\begin{proof}
	The measure $\mu_H$ is obtained from the Patterson-Sullivan measure $\sigma_0$ on $\partial \tilde M$ as
	$$
	\dd\mu_H (v)= e^{\delta \beta_{v_+}(0, \pi(v) ) } \dd\sigma_0(v_+).
	$$
	Since $\sigma_0$ is finite and this factor is bounded on bounded sets, $\mu_H$ is locally finite. We also know that $\sigma_0$ has no point masses, so neither does $\mu_H$.
\end{proof}

The orientation of $M$ induces an orientation on each horocycle $H$, so there are well defined positive and negative directions. A vector $v\in H$ divides the horocycle into two infinite intervals, we write $H_R(v)$ for the one in the positive direction and $H_L(v)$ for the other.  Next, we give some conditions on a subset of $\Sigma $ that allow us to define the Margulis parametrization.

\begin{proposition}
	\label{horocyclic_flow} Let $M$ be a compact oriented nonpositively curved rank $1$ surface.
	Let $\tilde{\Sigma}$ be a $\Gamma$-invariant Borel subset of $\TMu$ which is a union of horocycles. We assume that for every horocycle $H\subset{\tilde{\Sigma}}$,
	\begin{enumerate}[label={(\roman*)}]
		\item the measure $\mu_H$ is of full support in $H$,
		\item for one (hence for all) vector $v\in H$, the half horocycles $H_R(v)$ and $H_L(v)$ have infinite measure.
	\end{enumerate}
	Then there exists a horocyclic flow $\tilde{h}_s^M$ on $\tilde{\Sigma}$ such that for all $v\in\tilde{\Sigma}$ and $s\in\reals$, we have
	$$
	\mu_{H(v)}((v,\tilde{h}_s^M(v)))=|s|.
	$$
	Moreover, the flow $\tilde{h}_s^M$ satisfies, for every Borel subset $A$ of the horocycle $H(v)$,
	$$
	\mu_{H(v)}(A)=\Leb(\{s\in \reals \, | \,\tilde{h}_s^M(v)\in A\}).
	$$
\end{proposition}

\begin{proof}
	Let $H$ be a horocycle in $\tilde{\Sigma}$ and $v\in H$. We consider the function $m_v: H\to \reals$ defined by $m_v(w)=\pm \mu_{{H(v)}}((v,w))$ if $w\in H_{R/L}(v)$ and $m_v(v)=0$. It is continuous because of Proposition \ref{continuity}, strictly increasing because $\mu_H$ has full support in $H$ and is surjective because both half-horocycles have infinite measure. Since $H$ is homeomorphic to $\reals$, the map $m_v$ is a homeomorphism. We then define $\tilde{h}^M_s(v)=m_v^{-1}(s)$ for all $s\in\reals$. 
	
	It is clear that the orbit of $v$ is the horocycle $H$. Also, $\mu_H$ has no point masses by Lemma \ref{point mass}, so we have $m_{\tilde{h}_s^M(v)}(w)=m_v(w)-s$. This leads to the additive property $$\tilde{h}_s^M\circ \tilde{h}_{s'}^M=\tilde{h}_{s+s'}^M.$$ 
	To see the equality of measures, we observe that
	\begin{equation}\label{proof_measure}
		\mu_H((v,\tilde{h}_s^M(v)))=|m_v(\tilde{h}_s^M(v))|=|m_v(m_v^{-1}(s))|=|s|,
	\end{equation}
	from the definition of the flow. If we take the pullback $m_v^*\Leb$ of the Lebesgue measure by $m_v$, we can see that $m_v^*\Leb$ coincides with $\mu_H$ on the intervals $(w_1,w_2)$, so they are equal. Actually, if we set $w_1=\tilde{h}_{s_1}^M(v)$ and $w_2=\tilde{h}_{s_2}^M(v)$,
	\begin{align*}
		\Leb(m_v((w_1,w_2)))
		& =\Leb((s_1,s_2))=|s_2-s_1| \\
		&  	=\mu_H((w_1,\tilde{h}_{s_2-s_1}^M(w_1)))=\mu_H((w_1,w_2)),
	\end{align*}
	thanks to (\ref{proof_measure}).
	Then for every Borel subset $A$ of $H$, $\mu_{H}(A)=\Leb(m_v(A))$. But $s\in m_v(A)$ if and only if $\tilde{h}_{s}^M(v)=m_v^{-1}(s)\in A$. This proves the second property.
	
	It remains to prove the continuity of the flow as a map from $\reals \times \tilde{\Sigma}\to \tilde{\Sigma}$. Fix a couple $(s,v)\in \reals \times \tilde{\Sigma}$ and consider a sequence $((s_k,v_k))_k$ of elements of $  \reals\times\tilde{\Sigma}$ converging to $(s,v)$. We need to show that $\tilde{h}^M_{s_k}(v_k)$ converges to $\tilde{h}^M_s(v)$. We assume $s\ge 0$, the other case being done analogously. We know that the horocycles $H(w)$ depend continuously on $w$, so for each $k$ there exists a vector $w_k\in H(v_k)$ such that the sequence $\{w_k\}_k$ converges to $\tilde{h}^M_s(v)$. By Lemma \ref{continuity}, we know that $\mu_{H^u(v_k)}((v_k,w_k))$ converges to $\mu_{H^u(v)}((v,\tilde{h}^M_s(v)))={s}$ when $k$ tends to infinity. We deduce then that the measures of the intervals $(w_k,\tilde{h}^M_{s_k}(v_k))$ go to $0$. 
	
	We claim that the distance between $w_k$ and $\tilde{h}^M_{s_k}(v_k)$ tends to $0$. Assume, contrary to our claim, that, for some $\varepsilon>0$, there is a subsequence $k_i$ such that the Riemannian distance $d_1(w_{k_i},\tilde{h}^M_{s_{k_i}}(v_{k_i}))$ is greater than $\varepsilon$. Let us consider the points $\tilde{h}^L_{\varepsilon}(w_{k_i})$, which are in the interval $(w_{k_i},\tilde{h}^M_{s_{k_i}}(v_{k_i}))$. So the $\mu_{H(v_{k_i})}$-measure of $ (w_{k_i},\tilde{h}^L_{\varepsilon}(w_{k_i}))$ also tends to $0$. In the limit, we have $$\mu_{H(v)}((\tilde{h}^M_{s}(v),\tilde{h}^L_{\varepsilon}(\tilde{h}^M_{s}(v))))=0$$ thanks again to the continuity of the measures. This is a contradiction because $\mu_{H(v)}$ has full support in $H(v)$ by hypothesis. Finally, since $w_k$ converges to $\tilde{h}^M_s(v)$, the sequence $\tilde{h}^M_{s_k}(v_k)$ also converges to this point.
	
\end{proof}

The main result of this section is the following: there is a bijection between the set of finite invariant measures of the flow $h^L_s|_\Sigma$ and that of $h^M_s$.

\begin{proposition}
	Let $\tilde \Sigma$ be a subset of $T^1 \tilde M$ satisfying the hypothesis of Proposition \ref{horocyclic_flow}. Assume that there are at least two distinct stable manifolds of vectors in $\tilde \Sigma$ that do not have rank $2$ vectors. Consider the horocyclic flow $h_s^M$ on the set ${\Sigma}\subset T^1M$ there defined. The map
	 $\Phi: \mes({h}_s^L|_{{\Sigma}})\to \mes({h}_s^M)$ sends finite measures to finite measures, and infinite measures to infinite measures.
\end{proposition}
\begin{proof}
	We first show that the image of a finite measure is finite.
	Let $\mu\in \mes(\tilde{h}_s^L|_{\tilde{\Sigma}})$ be a $\Gamma$-invariant measure whose projection to $\Sigma\subset T^1M$ is finite and let $\mu'=\Phi(\mu)$ be its image in $\mes(\tilde{h}_s^M)$. 
	
	Let $D\subset \TMu$ be a compact fundamental domain for the action of $\Gamma$. We want to show that $\mu'(D\cap \tilde{\Sigma})$ is finite. Let $v_1, \, v_2\in \tilde \Sigma $ two vectors such that $\tilde{W}^{ws}(v_1)$ and $\tilde{W}^{ws}(v_2)$ are distinct sections of the horocyclic flow (i.e. they do not contain any rank $2$ vector). Take disjoint open neighborhoods $A,B$ in $\partial \tilde M$ of the points $v_{1+},v_{2+}$. By the continuity of the projection to the boundary we know that the sets
	$$
	\{w\in D \,|\,w_- \notin A\},\quad 	\{w\in D \,|\,w_- \notin B\}
	$$ 
	are closed in $D$, therefore compact. They form a cover of $D$. This reduces the proof to showing that $\mu'(K\cap \tilde{\Sigma})$ is finite where $K$ is a compact set such that $v_{0+} \notin K_-$, where $v_0$ is any vector of $\tilde \Sigma$ whose stable manifold has no rank $2$ vectors. Let us fix such a vector $v_0$ until the end of the proof.
	
	By lemma \ref{section}, for every $w\in K$ there exists a number $s(w)$ such that $\tilde{h}_{s(w)}^L(w)\in \tilde{W}^{ws}(v_0)$ (Figure \ref{fig:pic2}). We see $s(w)$ as a continuous function from $K$ to $\reals$. It is then bounded by some constant $S>0$. The function $w\mapsto \tilde{h}_{s(w)}^L(w)$ is also continuous, so the projection of $K$ to $\tilde{W}^{ws}(v_0)$ is a compact set, denoted by $L$. The following inclusion holds,
	$$
	K\subset \tilde{h}^L_{[-S,S]}(L).
	$$
	
	\begin{figure}[h]
		\centering
		\includegraphics{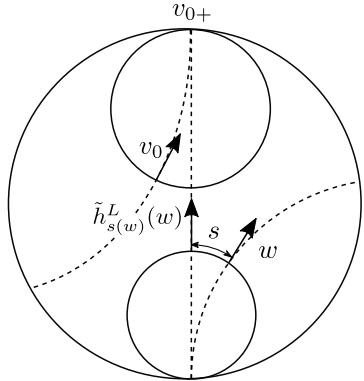}
		\caption{ \label{fig:pic2}
			Definition of $s(w)$.}
	\end{figure}
	
	The problem is reduced now to showing that sets of the form $\tilde{h}^L_{[-S,S]}(L)\cap \tilde{\Sigma}$ have finite $\mu'$-measure. Recall that our measure $\mu$ is the product of some measure $\mu_{\tilde{W}^{ws}(v_0)\cap \tilde{\Sigma}}$ on $\tilde{W}^{ws}(v_0)\cap \tilde{\Sigma}$ and the Lebesgue measures on the horocycles. Subsequently, 
	$$
	\mu(\tilde{h}^L_{[-S,S]}(L)\cap \tilde{\Sigma})=2S \cdot  \mu_{\tilde{W}^{ws}(v)\cap \tilde{\Sigma}}(L\cap \tilde{\Sigma}).
	$$ 
	Since $\tilde{h}_{[-S,S]}(L)$ is covered by a finite number of images of $D$ by elements of $\Gamma$, and $D\cap \tilde{\Sigma}$ has finite $\mu$-measure by hypothesis, then the left hand side of the previous equation is finite, and so is $\mu_{\tilde{W}^{ws}(v)\cap \tilde{\Sigma}}(L\cap \tilde{\Sigma})$.
	
	The measure $\mu'$ of the set $\tilde{h}^L_{[-S,S]}(L)\cap \tilde{\Sigma}$ can also be decomposed as 
	$$
		\mu'(\tilde{h}^L_{[-S,S]}(L)\cap \tilde{\Sigma})=\int_{L\cap \tilde{\Sigma}}\int_{s^M(v,-S)}^{s^M(v,S)} \dd s \,\dd \mu_{\tilde{W}^{ws}(v)\cap \tilde{\Sigma}}(u),
	$$
	where $s^M$ is the change of time between the Lebesgue and the Margulis flows. If the quantities $s^M(v,-S)$ and $s^M(v,S)$ are bounded on the set $ L\cap \tilde{\Sigma}$, then the last integral is bounded and we obtain the desired result.
	
	We have the equality $$s^M (v,\pm S)= \pm \mu_{H(v)}((v,\tilde{h}_{\pm S}^L(v))).$$ Thanks to the continuity of the measure on the horocycles, the functions $$v\mapsto \mu_{H(v)}((v,\tilde{h}_{\pm S}^L(v)))$$ are continuous (and globally defined). So they are bounded on $L$ because of the compactness. This completes the proof of the first implication.
	
	Let $\mu\in\mes(\tilde{h}_s^L|_{\tilde{\Sigma}})$ any $\Gamma$-invariant measure and assume that its image $\mu'=\Phi(\mu)$ in $\mes(\tilde{h}_s^M)$ induces a finite measure on the quotient $\Sigma\subset T^1 M$. We need to show that the $\mu$-measure of $D\cap \tilde{\Sigma}$ is finite, where $D$ is a compact fundamental domain.
	Similarly to the first situation we can reduce the problem to showing that, if $S>0$, $L$ is a compact subset of $\tilde{W}^{ws}(v)$ and $v\in \tilde \Sigma$ is a vector whose stable manifold has no rank $2$ vectors, then  the sets of the form $\tilde{h}^L_{[-S,S]}(L)\cap \tilde{\Sigma}$ have finite $\mu$-measure. Since $$\mu(\tilde{h}^L_{[-S,S]}(L)\cap \tilde{\Sigma})=2S \cdot  \mu_{\tilde{W}^{ws}(v)\cap \tilde{\Sigma}}(L\cap \tilde{\Sigma}),$$ this is equivalent to showing that $L\cap \tilde{\Sigma}$ has finite $\mu_{\tilde{W}^{ws}(v)\cap \tilde{\Sigma}}$-measure. 
	
	We know that $\tilde{h}^L_{[-S,S]}(L)\cap \tilde{\Sigma}$ has finite $\mu'$-measure, because it can be covered by finitely many images of $D$, and $D\cap \tilde \Sigma$ has finite measure. 
	This measure is the integral of the function $s^M(v,S)-s^M(v,-S)$ over $L\cap \tilde{\Sigma}$ with respect to $\mu_{\tilde{W}^{ws}(v)\cap \tilde{\Sigma}}$. But the function  $s^M(v,S)-s^M(v,-S)$ is strictly positive because $S>0$. So the measure $\mu_{\tilde{W}^{ws}(v)\cap \tilde{\Sigma}}(L\cap \tilde{\Sigma})$ is finite, otherwise we would obtain an infinite integral. This proves that $\mu(\tilde{h}^L_{[-S,S]}(L))$ is finite, as we wanted. 
\end{proof}

\begin{corollary}
	Let $ M$ be an oriented rank $1$ compact connected Riemannian surface with nonpositive curvature and let $\Sigma$ be a subset of $T^1M$ whose lift $\tilde \Sigma$ to $T^1 \tilde M$ satisfies the hypothesis of Proposition \ref{horocyclic_flow}.
	Then the map $\mu \mapsto \Phi(\mu)/\Phi(\mu)(\Sigma)$ is a bijection between $\prob(h_s^L|_{{\Sigma}})$ and $\prob(h_s^M)$.

\end{corollary}

	We finally apply this result to the subset $\Sigma_0$ of $T^1M$ defined as the union of horocycles containing a rank $1$ $g_t$-recurrent vector,
	$$ \Sigma_0 =\cup_{v\in Rec\,\cap R_1}H(v).$$
	This set satisfies the hypothesis of \ref{horocyclic_flow}, so the Margulis parametrization of the horocyclic flow can be defined. In \cite[Theorem 3.6]{Burniol21}, we proved that the Margulis flow on $\Sigma_0$ is uniquely ergodic. Thanks to the work in this paper, we now know that the same holds for the Lebesgue parametrization.

\begin{theorem}
	Let $M$ be an oriented rank $1$ compact connected Riemannian surface with nonpositive curvature. In restriction to $\Sigma_0$, the flow given by the parametrization by arc length of the horocycles is uniquely ergodic.
\end{theorem}

\section{Unique ergodicity on compact surfaces without flat strips}\label{Section4}

We now look back to Proposition \ref{horocyclic_flow}. There seem to be two difficulties in defining the Margulis parametrization on certain types of horocycles. One is a rather technical difficulty, the fact that half-horocycles have infinite measure. We are not sure to what extent this could fail. Both in \cite{Burniol21} and here, we have worked under the hypothesis that it is true. In contrast, the fact that a horocycle $H$ contains an interval of $\mu_H$-zero measure is a clear obstruction to the definition of the parametrization on $H$. This in fact can only happen if the interval consists of rank $2$ vectors. This phenomenon is produced by flat strips.

A \textit{flat strip} on the universal cover is a totally geodesic submanifold isometric to the space $\reals\times [0,r]$ for some $r>0$. Given a horocycle $H$ in $T^1\tilde M$, an interval $[v,w]\subset H$ consists only of rank $2$ vectors if and only if the geodesics generated by the vectors in $[v,w]$ form a flat strip. In this case, we say that $H$ cuts a flat strip.

Next we prove that the Margulis parametrization is defined on the whole unit tangent bundle if $M$ has no flat strips. With the help of this parametrization, we show that the horocyclic flow is uniquely ergodic using standard techniques.
\begin{proposition}
	Let $M$ be an oriented nonpositively curved compact surface without flat strips. Then for every horocycle $H$ in $T^1\tilde{M}$ the measure $\mu_H$ is of full support in $H$ and the $\mu_H$-measure of each  half-horocycle is infinite.
\end{proposition}
\begin{proof}
	Every interval $(v,w)$ on the horocycle $H$ contains a rank $1$ vector. Otherwise, all the vectors in $[v,w]$ would be of rank $2$, so the curvature would vanish everywhere on the geodesics they generate, and $v$ and $w$ would bound a flat strip.
	
	Since the rank $1$ set is open, then $(v,w)$ contains an interval of rank $1$ vectors. Now we use the fact that $ \mu_H$ is a positive function times the projection of $\sigma_0$ on the rank $1$ vectors (Corollary \ref{simplification_measure}). Recall that $\sigma_0$ is supported on the limit set, which is the whole boundary $\partial \tilde M$ because $M$ is compact, so the measure of $(v,w)$ is strictly positive. This proves that $\mu_H$ is fully supported.
	
	To prove the infiniteness of the measures $\mu_H$ on half-horocycles, we consider the map from $T^1M$ to $\reals$ which sends $v$ to the $\mu_{H(v)}$-measure of the horocyclic ball of center $v$ and radius $1$. This map is continuous, by Proposition \ref{continuity}, and $T^1M$ is compact, so it attains an absolute minimum $\alpha\in\reals$. But the map is everywhere strictly positive, so the constant $\alpha$ is strictly positive too. Now, a half-horocycle $H_{R/L}$ in the tangent space of the universal cover $\tilde{M}$, contains infinitely many disjoint unstable balls of radius $1$, each of them with measure at least $\alpha>0$. We conclude that the $\mu_H$-measures of the half-horocycles $H_{R/L}$ are infinite.	
\end{proof}

This result ensures that the conditions of Proposition \ref{horocyclic_flow} are satisfied on the whole unit tangent bundle. The Margulis parametrization can be defined everywhere for the class of nonpositively curved compact surfaces without flat strips. It induces a horocyclic flow $h_s^M$ on $T^1M$.
We deduce the main result of this section, namely the unique ergodicity of the horocyclic flow, thanks to the good properties of this parametrization.

\begin{theorem}\label{thm_un_nostrips}
	Let $M$ be an orientable nonpositively curved compact surface without flat strips. Then there is a unique Borel probability measure on $T^1M$ invariant by $h_s^M$. This measure is a constant multiple of the Bowen-Margulis measure. 
\end{theorem}

\begin{proof}
We observe that the Bowen-Margulis measure $\mu_{BM}$ is invariant by both the geodesic flow $g_t$ and the Margulis horocyclic flow $h^M_s$. The expanding property of the measures on the horocycles translates to the commuting property $g_t\circ h^M_s =h^M_{se^{\delta t}}\circ g_t$ between the flows. In this situation there is an argument of Coudène to show the unique ergodicity of $h^M_s$ when $\mu_{BM}$ is absolutely continuous with respect to the weak stable foliation \cite{Coudene1}. The geodesic flow of a nonpositively curved surfaces without flat strips does not meet this last requirement because there is not a local product structure on non-hyperbolic regions.	Indeed, weak stable manifolds are tangent to unstable horocycles on rank $2$ vectors. We thus need to adapt Coudène's argument.

Fortunately, the absolute continuity of $\mu_{BM}$ with respect to the weak stable foliation only intervenes in the proof of the equicontinuity of averages along a horocycle pushed by the geodesic flow. The latter fact can be shown in our case using the disintegration of $\mu_{BM}$ on the boundary $\partial \tilde M$.

\begin{lemma} Let $M$ be an orientable nonpositively curved compact surface without flat strips. Let $f:T^1M \to \reals$ be a continuous function. For every $v\in T^1M$ and $R>0$, write  
	\begin{equation*}
		M_R(f)(v):=\frac{1}{R}\int_0^R {f}(h_s^M(v))ds.
	\end{equation*}
	Then the family of functions $\{M_1(f\circ g_t)\}_{t>0}$ is equicontinuous at every point $v\in T^1 M$.
\end{lemma}

\begin{proof}
	We will rather work on the universal cover. Let $f$ be a real continuous $\Gamma$-invariant function on $T^1 \tilde M$. Since $\Gamma$ is cocompact, the absolute value of $f$ is bounded by a real constant $C>0$. From the definition of $\tilde{h}_s^M$, the average of $f$ is 
	\begin{equation*}
		M_1(f\circ g_t)(v)=\int_{[v,\, \tilde{h}_1^M(v)]}f\circ g_t \, \dd \mu_{H(v)},		
	\end{equation*}
	which is also written as 
	$$
	M_1(f\circ g_t)(v) =	\int_{\partial \tilde M} \mathbf{1}_{[v_+,\tilde{h}_1^M(v)_+]\setminus \tilde{S}_+}\cdot  f\circ g_t \circ P_{H(v)}^{-1} \cdot 
	\phi_{H(v)}
	\,\dd \sigma_0
	$$
	by Corollary \ref{simplification_measure}.
	
	Given two vectors $v,w\in T^1\tilde M$, we observe 
	\begin{align*}
		|M_1&(f\circ g_t)(v) - M_1(f\circ g_t)(w)| \le  \\
		& \le \int_{\partial \tilde M} | (\mathbf{1}_{[v_+,\tilde{h}_1^M(v)_+]\setminus \tilde{S}_+} \phi_{H(v)} - 
		\mathbf{1}_{[w_+,\tilde{h}_1^M(w)_+]\setminus \tilde{S}_+} \phi_{H(w)} )
		\cdot  f\circ g_t \circ P_{H(w)}^{-1} |
		\,\dd \sigma_0  \\
		& + \int_{\partial \tilde M} | \mathbf{1}_{[v_+,\tilde{h}_1^M(v)_+]\setminus \tilde{S}_+} \phi_{H(v)} 
		\cdot  (f\circ g_t \circ P_{H(v)}^{-1} - f\circ g_t \circ P_{H(w)}^{-1}) |
		\,\dd \sigma_0 \\
		& \le C \int_{\partial \tilde M} | \mathbf{1}_{[v_+,\tilde{h}_1^M(v)_+]\setminus \tilde{S}_+} \phi_{H(v)} - 
		\mathbf{1}_{[w_+,\tilde{h}_1^M(w)_+]\setminus \tilde{S}_+} \phi_{H(w)}  |
		\,\dd \sigma_0
		\tag{$\ast$}  \\
		& + \int_{[v,\tilde{h}^M_1(v)]} |f\circ g_t - f\circ g_t \circ P_{H(w)}^{-1}\circ P_{H(v)}| \dd \mu_{H(v)}.
		\tag{$\ast\ast$} 
	\end{align*}
	The term ($\ast$) is independent of $t$ and tends to $0$ when $w$ tends to $v$. This is because the function $\mathbf{1}_{[w_+,\tilde{h}_1^M(w)_+]\setminus \tilde{S}_+} \phi_{H(w)}$ converges almost surely to $\mathbf{1}_{[v_+,\tilde{h}_1^M(v)_+]\setminus \tilde{S}_+} \phi_{H(v)}$. The function $P^{-1}_{H(w)}\circ P_{H(v)}$ in the term ($\ast\ast$) is defined at least on the set $P_{H(v)}^{-1}(P_{H(w)}(H(w))\setminus \tilde{S}_+)$, which has full measure. The map 
	$$(w,u)\mapsto P^{-1}_{H(w)}\circ { P}_{H(v)}(u)= {\bar P}^{-1}(w_-,u_+,\beta_{w_-}(0,\pi(w)))$$
	is continuous on its domain. 
	
	Given $\delta>0$, we can take $w$ in a neighborhood of $v$ such that the distance between $u$ and $P^{-1}_{H(w)}\circ P_{H(v)}(u)$ is less than $\delta $ for all $u\in [v,\tilde{h}^M_1(v)]$ where the function is defined. Since $u$ and $P^{-1}_{H(w)}\circ P_{H(v)}(u)$ are on the same weak stable manifold, the distance between them is non-increasing when we apply $g_t$. The distance between $g_t(u)$ and $g_t \circ P^{-1}_{H(w)}\circ P_{H(v)}(u)$ is therefore less than $\delta $ for all positive $t$. In this way, the term ($\ast\ast$) is bounded by the modulus of continuity $\omega_f(\delta )$ of $f$, which goes to $0$ if $\delta \to 0$ because $f$ is uniformly continuous. This proves the equicontinuity of the functions $\{M_1(f\circ g_t)\}_{t>0}$ at $v$.
	
\end{proof}

It is straightforward to see that nonwandering rank $1$ vectors are contained in the support of $\mu_{BM}$ from the definition. But every vector is nonwandering and the rank $1$ set is dense in $T^1M$, because $M$ is compact. So $\mu_{BM}$ is fully supported. It is also known that the horocyclic foliation is transitive \cite[Theorem 5.2]{Eberlein73a}. These are the remaining ingredients of Coudène's theorem. 

The rest of the proof goes verbatim to the one by Coudène. We explain it for the sake of completeness. We use the equicontinuity to apply the Arzelà-Ascoli theorem to $\{M_1(f\circ g_t)\}_{t>0}$. Hence, $\{M_1(f\circ g_t)\}_{t>0}$ is relatively compact in the space of continuous functions on $T^1M$ endowed with the uniform topology. Let $\bar f$ be an accumulation point of this family, so there is a sequence $t_k\to +\infty$ such that $M_1(f\circ g_{t_k})$ converges to $\bar f$. 

The commuting relation between $g_t $ and $h_s^M$ yields the formula
\begin{equation}\label{commuting}
	M_1(f\circ g_t)= M_{e^{\delta t}} (f) \circ g_t.
\end{equation}
Moreover, by the Von Neumann ergodic theorem, $M_R(f)$ converges in $L^2(\mu_{BM})$ to an $h_s^M$-invariant function $Pf$. Combining both facts with the $g_t$-invariance of $\mu_{BM}$, we have
$$
||\bar f - Pf \circ g_{t_k}||_{2} \le  ||  M_1(f\circ g_{t_k}) - \bar f ||_{2}+ ||M_{e^{\delta t_k}} (f) -Pf ||_2.
$$
This inequality implies that $\bar f $ is a $L^2$ limit of $h_s^M$-invariant functions. So $\bar f$ is continuous and $h_s^M$-invariant, and in fact it is constant since $h_s^M$ has a dense orbit. This constant is $\int f\dd \mu_{BM}$.

Since $\{M_1(f\circ g_t)\}_{t>0}$ has a unique accumulation point, the quantity $M_1(f\circ g_t)$ converges uniformly to this accumulation point when $t$ goes to infinity. Using again (\ref{commuting}) and the fact that the accumulation point is constant, we deduce that $M_R(f)$ converges uniformly to the same constant when $R$ goes to infinity. This implies the unique ergodicity of $h_s^M$. 

\end{proof}

Finally, as a corollary of Theorem \ref{thm_un_nostrips}, we can deduce the unique ergodicity of the horocyclic flow for other parametrizations, for example by the arc-length.

\begin{corollary}
	Let $M$ be an orientable nonpositively curved compact surface without flat strips. The Lebesgue horocyclic flow $h_s^L$ on $T^1M$ is uniquely ergodic.
\end{corollary}

\bibliographystyle{plain}
\bibliography{general_library}
\end{document}